\documentclass[12pt]{amsart}

\usepackage{amsmath,amsfonts,amsthm,amsopn,cite,mathrsfs}
\usepackage{epsfig,verbatim}
\usepackage{subfigure}

\newcommand{\ft}{Fourier transform}

\newcommand{\fif}{if and only if}
\newcommand{\tfs}{time-frequency shift}

\newtheorem{tm}{Theorem}
\newtheorem{lemma}[tm]{Lemma}

\newtheorem{cor}[tm]{Corollary}

\newcommand{\rems}{\noindent\textsl{REMARKS:}}
\newcommand{\rem}{\noindent\textsl{REMARK:}}

\newcommand{\beqa}{\begin{eqnarray*}}
\newcommand{\eeqa}{\end{eqnarray*}}

\DeclareMathOperator*{\supp}{supp}

\newcommand{\field}[1]{\mathbb{#1}}
\newcommand{\bR}{\field{R}}        
\newcommand{\bN}{\field{N}}        
\newcommand{\bZ}{\field{Z}}        
\newcommand{\bC}{\field{C}}        
\newcommand{\vf}{\varphi}

 \def\cF{\mathcal{F}}              

 \def\cG{\mathcal{G}}
 \def\cL{\mathcal{L}}
 \def\cM{\mathcal{M}}

 \def\cO{\mathcal{O}}

 \def\cX{\mathcal{X}}

\def\rd{\bR^d}

\def\zd{\bZ^d}

\def\lrd{L^2(\rd)}

\def\intrd{\int_{\rd}}

\def\<{\left<}
\def\>{\right>}

\def\inv{^{-1}}

\def\mv1{M_v^1}

\newcommand{\vs}{\vspace{3 mm}}

\newcommand{\gabsy}{\cG (g,\alpha, \beta )}

\newcommand{\arro}{\Rightarrow}
\newcommand{\rplus}{\bR _+^2}

\begin{document}
\begin{abstract}
Let $g$ be a totally positive function of finite type, i.e.,
$\hat{g}(\xi ) =  \prod_{\nu=1}^{M}(1+2\pi i \delta_\nu \xi )\inv
  $ for $\delta _\nu \in \bR $ and $M\geq 2$. Then the set $\{e^{2\pi i \beta l t} g(t-\alpha k)
  : k,l \in \bZ \}$ is  a frame for $L^2(\bR  )$, \fif\ $\alpha  \beta
  <1$. This result is a first positive
  contribution to a conjecture of I.\ Daubechies from 1990. So far the
  complete characterization of lattice parameters $\alpha, \beta $
  that generate a frame  has been
  known for only six window functions $g$. Our  main result now
  provides an uncountable  class of functions.  As a
  byproduct of the proof method  we derive new sampling theorems in
  shift-invariant spaces and obtain  the correct Nyquist rate.
\end{abstract}

\title{Gabor Frames and Totally Positive Functions}
\author{Karlheinz Gr\"ochenig}
\address{Faculty of Mathematics \\
University of Vienna \\
Nordbergstrasse 15 \\
A-1090 Vienna, Austria}
\email{karlheinz.groechenig@univie.ac.at}
\author{Joachim St\"ockler}
\address{TU Dortmund \\
Vogelpothsweg 87 \\
D-44221 Dortmund }
\email{joachim.stoeckler@math.tu-dortmund.de}

\subjclass[2000]{}
\date{}
\thanks{K.\ G.\ was
  supported in part by the  project P22746-N13  of the
Austrian Science Foundation (FWF)}
\maketitle

\section{Introduction}

The fundamental problem  of Gabor analysis is to determine triples
$(g,\alpha ,\beta )$ consisting of an $L^2$-function $g$ and lattice
parameters $\alpha ,\beta >0$, such that the set of functions
$\cG (g,\alpha,\beta ) = \{e^{2\pi i \beta l t} g(t-\alpha k) : k,l \in \bZ \}$ constitutes a
frame for $L^2(\bR  )$. Thus the fundamental problem is to determine
the set (the \emph{frame set})
\begin{equation}
  \label{eq:7}
  \cF (g) = \{ (\alpha , \beta )\in  \bR _+^2 : \gabsy \,
  \text{ is a frame } \} \, .
\end{equation}

It is stunning how little is known about the nature of the set $\cF
(g)$, even  after twenty years of Gabor analysis.  The famous Janssen
tie~\cite{janssen02b} shows that the set $\cF(g) $ 
can be arbitrarily complicated, even for a
``simple'' function such as  the characteristic function $g= \chi _I$
of an interval.

 Under mild conditions,
precisely, if $g$ is in the Feichtinger algebra $M^1$, then the set
$\cF (g) $ is open in $\bR _+^2$~\cite{FK04}.  Furthermore, if $g\in M^1$,
then $\cF (g)$ contains a neighborhood $U$  of $0$ in $\bR _+^2$. Much
effort has been spent to improve the analytic estimates and make this
neighborhood as large as possible~\cite{BC02a,daubechies90,ron-shen97}. The fundamental density
theorem asserts that $\cF (g) $ is always a  subset of $\{(\alpha,
\beta )\in \rplus : \alpha
\beta \leq 1\}$~\cite{book,daubechies92,heil07}. If $g\in M^1$, then  a subtle version of the
uncertainty principle,  the so-called  Balian-Low theorem,
 states that
$\cF (g) \subseteq \{(\alpha, \beta ):   \alpha
\beta < 1\}$~\cite{BHW95,CP06}.  This means that $\{(\alpha, \beta ):  \alpha
\beta \leq 1\}$ is the maximal set that can occur as a frame set $\cF
(g)$.

Until now, the catalogue of windows $g$ for which $\cF (g)$ is
completely known, consists of the following functions:  if $g$ is either the  Gaussian $g(t)= e^{-\pi
  t^2}$,  the
hyperbolic secant $g(t) =  (e^{t} + e^{-t})\inv $,  the exponential
function $e^{-|t|}$,  then  $\cF (g) = \{
(\alpha ,\beta )\in \rplus : \alpha \beta <1\}$;  if $g$ is the  one-sided
exponential function $g(t) = e^{-t} \chi _{\bR ^+} (t)$,  then $\cF (g) = \{
(\alpha ,\beta )\in \rplus:  \alpha \beta \leq 1\}$.
In addition, the dilates of these functions and their  Fourier
transforms, $g(t) = (1+2\pi i t)\inv $ and $g(t) = (1+4\pi ^2 t^2)\inv
$,  also have the same frame set.
The case of the Gaussian was solved independently by Lyubarski~\cite{lyub92}
and Seip~\cite{seip92} in 1990 with methods from complex analysis in response to a conjecture by
Daubechies and Grossman~\cite{DG88}; the case of the hyperbolic secant can be
reduced to the Gaussian with a trick of Janssen and Strohmer~\cite{JS02},
the case of the  exponential functions is due to  Janssen~\cite{janssen96,janssen02c}.
We note that in all these cases the necessary density condition
$\alpha \beta < 1$ (or $\alpha \beta \leq 1$) is also sufficient for
$\gabsy $ to generate a frame.

The example of the Gaussian lead Daubechies to conjecture that
 $\cF (g) = \{ (\alpha ,\beta )\in \rplus :  \alpha \beta <1\}$
 whenever $g$ is a positive function in $L^1$  with positive  Fourier
 transform in $L^1$~\cite[p.~981]{daubechies90}.  This conjecture was
 disproved in~\cite{janssen96a}. 

Surprisingly, no alternatives to  Daubechies' conjecture  have  been
formulated so far. In this paper we deal with a modification of
Daubechies' conjecture and prove that the frame set of  a  large class
of functions is indeed the maximal set $\cF (g) = \{ (\alpha, \beta
)\in \rplus : \alpha \beta <1\}$. 

This breakthrough is possible by combining ideas from Gabor analysis,
approximation theory and spline theory, and sampling theory.
The main observation is that
all  functions above --- the Gaussian, the hyperbolic secant, and the
 exponential functions  --- are \emph{totally positive
  functions}. This means that     for every two sets of increasing
real numbers $
   x_1<x_2<\cdots<x_N$ and $ y_1<y_2<\cdots<y_N$ , $ N\in \bN $,
the determinant of the matrix $\left[ g(x_j-y_k)\right]_{1\le j,k  \le
  N} $ is non-negative. 

Indeed, for a large class of totally positive functions to be defined
in \eqref{eq:totposI} we will  determine the set $\cF (g)$ completely.
Our main result is the following:

\begin{tm}\label{tmm1}
  Assume that $g\in L^2(\bR)$ is a totally positive function of finite type $\geq
  2$. Then $\cF (g)= \{ (\alpha, \beta ) \in \bR _+^2: \alpha \beta
  <1\}$. In other words,   $\gabsy $ is a
  frame, \fif\ $\alpha \beta <1$.
\end{tm}

This theorem increases the number of functions with known frame set
from six to uncountable. We will see later that the totally positive
functions of finite type can be parametrized by a countable number of
real parameters, see~\eqref{eq:totposI}.
Among the examples of totally positive functions of finite type  are the
two-sided exponential $e^{-|t|}$ (already known), the truncated power functions $g(t)
= e^{-t} t^r \chi _{\bR _+}$ for $r\in \bN $, the function $g(t) = (e^{-at} - e^{-bt})
\chi _{\bR ^+}(t)$ for $a,b>0$, or the asymmetric exponential $g(t) =
e^{at} \chi _{\bR ^+}(-t) +  e^{-bt} \chi _{\bR ^+}(t)$, and the
convolutions of totally positive functions of finite type. In addition
the class of $g$ such that $\cF (g)= \{ (\alpha, \beta ) \in \bR _+^2: \alpha \beta
  <1\}$   is invariant with respect to  dilation, \tfs s, and the \ft
  .  Since $\cG (g, \alpha ,
\beta ) = \cG (\hat{g}, \beta , \alpha )$, we obtain a complete
description of the frame set of the Fourier transforms of totally
positive functions. For instance, if $g(t) = (1+4\pi ^2 t^2)^{-n}$ for
$n\in \bN $, then $\cF ({g})= \{ (\alpha, \beta ) \in \bR _+^2: \alpha \beta
  <1\}$.

To compare with Daubechies' original conjecture, we note that every
totally positive and even function possesses a  positive Fourier
transform. Theorem~\ref{tmm1} yields a large class of functions for
which Daubechies' conjecture is indeed true. Furthermore,
Theorem~\ref{tmm1} suggests the modified  conjecture that the frame
set of every continuous totally positive function is $\cF ({g})= \{
(\alpha, \beta ) \in \bR _+^2: \alpha \beta   <1\}$.

Our main tool is a  generalization of the total positivity to
infinite matrices. We will show that an infinite matrix
of the form $\left[ g(x_j-y_k)\right]_{j,k \in
\bZ } $  possesses a left-inverse, when   $g$ is totally positive  and
some natural conditions hold for  the sequences
$(x_j)$ and $(y_k)$   (Theorem~\ref{main0}).

The analysis of Gabor frames and the ideas developed in the proof of
Theorem~\ref{tmm1}   lead to a surprising progress on  another open problem, namely
(nonuniform) sampling  in shift-invariant spaces. Fix a generator $g\in L^2(\bR
)$, a step-size $h>0$,   and consider the subspace of $L^2(\bR )$ defined by
$$
V_h(g) = \{ f \in L^2(\bR ): f= \sum _{k\in \bZ } c_k g(.-kh) \} \, .
$$
For the case $h=1$ we write $V(g)$, for short.
We  assume that the translates $g( . - k)$, $k\in \bZ $, form a Riesz
basis for $V(g)$ so that $\|f \|_2 \asymp \|c\|_2$.
Shift-invariant spaces are used as an attractive substitute of
bandlimited
functions in signal processing to model ``almost'' bandlimited
functions. See the survey~\cite{AG01} for an introduction to sampling
in shift-invariant spaces.  An important  problem that is  related to
the analog-digital conversion in signal processing  is the derivation of
sampling theorems for the space $V(g)$. We say that a set of sampling points $x_j $, ordered
linearly as $x_j < x_{j+1}$, is a set of sampling for $V_h(g)$, if there
exist constants $A,B>0$, such that
$$
A \|f\|_2^2 \leq \sum _{j\in \bZ } |f(x_j)|^2 \leq B \|f\|_2^2  \qquad
\text{ for all } f\in V_h(g) \, .
$$

As in the case of Gabor frames there are many qualitative sampling
theorems for shift-invariant spaces. Typical results  require high  oversampling rates.
They state that there exists  a $\delta >0 $
depending on $g$ such that every set with maximum gap $\sup _j
(x_{j+1} - x_j) = \delta $ is a set of sampling for $V(g)$\cite{AF98,AG01,Li07}. In most
cases $\delta $ is either not specified or too small to be of
practical use. The expected  result is that for $V(g)$  there exists a
Nyquist rate and that $\delta <1$ is sufficient. And yet,   the
only  generators for which the sharp   result
is known are the $B$-splines $b_n = \chi _{[0,1]} \ast \dots \ast \chi
_{[0,1]}$ ($n+1$-times). If the maximum gap $\sup _j
(x_{j+1} - x_j) = \delta $ satisfies $\delta <1$, then $\{ x_j\}$ is a
set of sampling for $V(b_n)$~\cite{AG00}.   (This
optimal result can also be proved for a generalization of splines, the
so-called ``ripplets''\cite{CGM}). It has been an open problem to
identify further classes of shift-invariant spaces for which the optimal
sampling results hold.

Here we will prove a similar  result for totally positive
generators.

\begin{tm}
  Let $g$ be a totally positive function of finite type and $\cX = \{x_j\}$
  be a  set with maximum gap $\sup _j
(x_{j+1} - x_j) = \delta $. If $\delta <1$, then $\cX $ is a set of
sampling for $V(g)$.
\end{tm}

The theorem will be a corollary of a  much more general
sampling theorem.

The paper is organized as follows: In Section~2 we discuss the tool
box for the Gabor frame problem. We will review some known  characterizations
of Gabor frames and derive some new criteria  that are more
suitable for our purpose. We then recall the main statements about
totally positive functions and prove the main technical theorem about
the existence of a left-inverse of the pre-Gramian  matrix. In
Section~3 we study Gabor frames and discuss some  open
problems that are raised by our new results.  In Section~4 we prove the sampling theorem.

\section{Tools}

\subsection{Characterizations of Gabor Frames}

There are many results about the structure of Gabor frames and
numerous characterizations of Gabor frames. In principle, one has to
check that one of the equivalent conditions for a set
$\gabsy $ to be  a frame is satisfied.  This task is almost always
difficult because it amounts to proving the invertibility of an
operator or a family of operators on an infinite-dimensional Hilbert
space.

In the following we summarize the most important characterizations of
Gabor frames. These are valid in arbitrary dimension $d$ and for
rectangular lattices $\alpha \zd \times \beta \zd $. We will  use the
notation $M_\xi f=e^{2\pi i \xi\cdot }f$ and
$T_yf=f(\cdot -y)$, $\xi,t \in \rd $,  such that
$$
   \gabsy =\{ M_{l\beta}T_{k\alpha}g : k,l\in \zd\}\, .
$$
Then $\gabsy$ is a frame of $L^2(\rd)$, if 
 there exist constants $A,B>0$, such that
$$
A\|f\|_2^2 \leq \sum _{k,l \in \bZ } |\langle f,
M_{l\beta}T_{k\alpha}g \rangle |^2 \leq B \|f\|_2^2 \qquad \text{for
  all } f\in L^2(\rd ) \, .
$$
If only the right-hand inequality is satisfied, then $\gabsy $ is called a
Bessel  sequence.

Following the fundamental work of Ron and Shen~\cite{ron-shen97} on
shift-invariant systems and Gabor frames, we define two
families of infinite matrices associated to a given  window function
$g\in \lrd $ and two lattice parameters $\alpha, \beta >0$.

The pre-Gramian matrix $P=P(x)$ is defined by the entries
\begin{equation}
  \label{eq:1}
  P(x)_{jk} = g(x+j\alpha - \tfrac{k}{\beta}),  \qquad j,k \in \zd \, .
\end{equation}
The Ron-Shen matrix is $G(x) = P(x)^* P(x)$ with the  entries
\begin{equation}
  \label{eq:2}
  G(x)_{kl} = \sum _{j\in \zd } g(x+j\alpha - \tfrac{l}{\beta}) \,
  \bar{g}(x+j\alpha - \tfrac{k}{\beta}),  \qquad k,l \in \zd \, .
\end{equation}

\begin{tm}[Characterizations of Gabor frames] \label{charone}
  Let $g\in \lrd$ and $\alpha , \beta >0$. Then the following
  conditions are  equivalent:

(i) The set $\gabsy $ is a frame for $\lrd $.

(ii) There exist $A,B >0$ such that the spectrum of almost every
Ron-Shen matrix $G(x)$ is contained in the interval $[A,B]$:
$$
\sigma (G(x)) \subseteq [A,B] \qquad \text{a.a. } x\in \rd \, .
$$

(iii) There exist $A,B >0$ such that
\begin{equation}
  \label{eq:4}
A\|c\|_2^2 \leq \sum _{j\in \zd } \big| \sum _{k\in \zd } c_k
g(x+j\alpha - \tfrac{k}{\beta})\big|^2 \leq B \|c\|_2^2  \quad
\text{a.a.} \,\, x\in \rd, c\in \ell ^2(\zd ) \, .
\end{equation}

(iv) There exists a so-called dual window $\gamma $, such that $\cG
(\gamma , \alpha , \beta )$ is a Bessel sequence and $\gamma $
satisfies the biorthogonality condition
\begin{equation}
  \label{eq:3}
\langle \gamma , M_{l/\alpha } T_{k/\beta } g\rangle = (\alpha \beta
)^d \delta _{k,0}\delta _{l,0}, \qquad \forall k,l \in \zd   \, .
\end{equation}
\end{tm}

\rems\ 1. Condition (iii) is only  a simple  reformulation of (ii), because
\eqref{eq:4} in different notation is just
$$
\|P(x)c\|_2^2 = \langle P(x)c, P(x)c\rangle = \langle G(x) c, c
\rangle \asymp \|c\|_2^2 \, .
$$

2. Condition (iii) as stated has been used  successfully by
Janssen~\cite{janssen96,janssen02c}.   The construction of a
dual window is underlying the proofs that $\gabsy $ is a frame for the
Gaussian $g(t) = e^{-\pi t^2}$ in~\cite{GL09} and for the one-sided
exponential $g(t) = e^{-t}\chi _{\bR ^+}(t)$~\cite{janssen96}.

3. Condition (iii) establishes a fundamental link between Gabor
analysis and the theory of  sampling in
shift-invariant spaces. It  says that the set  $x+\alpha \bZ $
is a set of sampling for the  shift-invariant space $V_h(g)$ with
generator $g$ and step-size $h=1/\beta$.
Thus $\gabsy $ is a frame for $\lrd $, \fif\ each set  $x+\alpha
\bZ $  is a set of sampling for $V(g)$ \emph{with uniform constants
independent of $x\in \rd $}.

In our analysis we will use another reformulation of condition (iii).

\begin{lemma} \label{l:pre}
  Let $g\in \lrd$ and $\alpha , \beta >0$. Then the following
  conditions are  equivalent:

(i) The set $\gabsy $ is a frame for $\lrd $.

(v) The set  of the pre-Gramians $\{P(x) \}$ is uniformly bounded on
$\ell ^2(\zd )$,  and  possess a uniformly bounded set of  left-inverses, i.e.,
there exist  matrices  $\Gamma (x), x\in \rd ,$ such
that
\begin{align}
  \Gamma (x) P(x) = I \qquad \text{ a.a. } x \,\, \in \rd \, ,\\
\|\Gamma (x) \|_{op} \leq C \qquad \text{ a.a. } x  \,\, \in \rd \, .
\end{align}

In this case,  the function $\gamma$ defined by $\gamma (x+\alpha j)
=  \beta  ^{d} \bar{\Gamma} _{0,j}  (x)$, where $x\in [0,\alpha)^d$ and $j\in\zd$,
or
equivalently
\begin{equation}
  \label{eq:9}
  \gamma (x) =  \beta ^{d} \sum _{j\in \zd } \bar{\Gamma} _{0,j}(x) \chi _{[0,\alpha
    )^d}(x-\alpha j) \, ,\quad x\in\rd,
\end{equation}
 satisfies the biorthogonality condition \eqref{eq:3}.
\end{lemma}

\begin{proof}

  $\mathbf{(i) \arro (v)}$. If $\gabsy $ is a frame, then by
  Theorem~\ref{charone}(ii)
$$ \sum _{j\in \zd } \big| \sum _{k\in \zd } c_k
g(x+j\alpha - \tfrac{k}{\beta})\big|^2 = \langle P(x) c , P(x)c\rangle
= \langle G(x) c, c\rangle  \asymp \|c\|_2^2 \, , $$
with bounds independent of $x$.
Consequently  $G(x)$ is bounded and invertible on $\ell^2(\zd)$.
Therefore   the operators $P(x)$ are uniformly  bounded on $\ell
^2(\zd  )$,  
 and we can define
$\Gamma (x) = G(x)\inv P^*(x)$. Then
$$\Gamma (x) P(x) =
\big((P^*(x)P(x))\inv P^*(x)\big) P(x) = \mathrm{Id},$$
 and
$$\|\Gamma (x)\|_{op} \leq \|G(x)\inv \|_{op} \|P(x)\|_{op} \leq A\inv B^{1/2}.$$

  $\mathbf{(v) \arro (ii)}$. Conversely, if $P(x)$ possesses a bounded left-inverse
  $\Gamma (x)$, then
$$
\|c\|_2^2 = \|\Gamma (x) P(x) c\|_2^2 \leq  \|\Gamma (x) \|_{op}^2
\|P(x) c\|_2^2 \le  C^2 \langle G(x) c , c \rangle \leq C^2
\|P (x) \|_{op}^2 \|c\|_2^2 \, ,
$$
and this implies condition  $\mathbf{(ii)}$ of Theorem~\ref{charone}.

  We next  verify that $\gamma$ as  defined
   in \eqref{eq:9}  satisfies the
  biorthogonality condition \eqref{eq:3}:
  \begin{align}
    \langle \gamma , M_{l/\alpha } T_{k/\beta } g\rangle &= \intrd
    \gamma (x) \bar{g}(x-k/\beta ) \, e^{-2\pi i l\cdot x/\alpha } \,
    dx \notag \\
&= \int _{[0,\alpha ]^d} \sum _{j\in \zd } \gamma  (x+\alpha j)
\bar{g}(x+\alpha j - k/\beta ) \,  \, e^{-2\pi i l\cdot x/\alpha } \,
    dx \notag \\
&= \beta ^{d} \int _{[0,\alpha ]^d} \sum _{j\in \zd } \bar{\Gamma} _{0,j}  (x)
\bar{g}(x+\alpha j - k/\beta ) \,  \, e^{-2\pi i l\cdot x/\alpha } \,
    dx \label{newdual} \\
&=  \beta ^{d} \int _{[0,\alpha ]^d} \delta _{k,0} e^{-2\pi i l\cdot x/\alpha } \,
    dx = (\alpha\beta )^d \delta _{k,0} \delta _{l,0} \, . \notag
  \end{align}
\end{proof}

The  function $\gamma$ in \eqref{eq:9} is a dual window of $g$, as
defined in
condition (iv) of Theorem~\ref{charone}, provided
that  $\mathcal{G}(\gamma,\alpha,\beta)$ is a Bessel sequence.
The following result gives a sufficient condition.

\begin{lemma}
  \label{l:pre1}
Assume that there exists a (Lebesgue measurable)  vector-valued function
$\sigma (x)$ from $\rd$ to $\ell ^2(\zd )$ with period $\alpha $, such that
\begin{equation}
  \label{eq:c6}
  \sum _{j\in \zd } \sigma _j(x) \bar{g} (x+\alpha j -
  \tfrac{k}{\beta})  = \delta _{k,0} \qquad \text{a.a. } x \in \rd  \, .
\end{equation}
If $\sum _{j\in \zd } \sup _{x\in [0,\alpha ]^d} |\sigma _j(x)| < \infty $,
then  $\gabsy $ is a frame. Moreover, with
$$\gamma (x) =  \beta ^{d} \sum _{j\in \zd } \sigma _{j}(x) \chi _{[0,\alpha
    )^d}(x-\alpha j),\qquad x\in\rd,$$
the set $\mathcal{G}(\gamma,\alpha,\beta)$   is a dual frame of $\gabsy$.
\end{lemma}

\begin{proof}
The  computation in~\eqref{newdual} shows that $\gamma $ satisfies the
biorthogonality condition \eqref{eq:3}.
The additional assumption implies that
$$
\sum _{k\in \zd } \sup _{x\in [0,\alpha ]^d} |\gamma (x+\alpha k)|
<\infty \, .
$$ Consequently,  $\gamma $ is in the
amalgam space $W(\ell ^1)$. This property guarantees that $\cG (\gamma
, \alpha ,\beta )$ is a Bessel system~\cite{walnut92}. Thus condition $(iv)$ of
Theorem~\ref{charone} is satisfied, and $\gabsy $ is a frame.  The
biorthogonality condition \eqref{eq:3} implies that
 $\mathcal{G}(\gamma,\alpha,\beta)$ is a  dual frame of  $\gabsy$.

\end{proof}

\subsection{Totally Positive Functions}

\noindent
The notion of totally positive functions was introduced in 1947 by I. J. Schoenberg
\cite{Schoenb1947a}.
A non-constant  measurable function $g:\bR\to\bR$
is said to be {\em totally positive}, if  it satisfies the
following  condition:
For every two sets of increasing real numbers
\begin{equation}\label{eq:numbers}
   x_1<x_2<\cdots<x_N,\qquad y_1<y_2<\cdots<y_N,\qquad N\in \bN,
\end{equation}
we have the inequality
\begin{equation}\label{eq:Dpos}
    D=\det \left[ g(x_j-y_k)\right]_{1\le j,k\le n} \ge 0.
\end{equation}
 Schoenberg \cite{Schoenb1951a} connected the   total
positivity  to a  factorization  of the (two-sided) Laplace transform of $g$,
\[
    \cL[g](s) = \int_{-\infty}^\infty e^{-st}g(t)\,dt =: \frac{1}{\Phi(s)}.
\]

\begin{tm}[Schoenberg \cite{Schoenb1947a}] \label{thm:schoen} The function $g:\bR\to\bR$ is totally positive, if and only if
its two-sided Laplace transform
exists in a strip $S=\{s\in \bC:\alpha<\text{Re}\,s<\beta\}$, and
\begin{equation}\label{eq:totpos}
   \Phi(s)=\frac{1}{\cL[g](s)} =
   Ce^{-\gamma s^2+\delta s}\prod_{\nu=1}^{\infty}(1+\delta_\nu s)e^{-\delta_\nu s},
\end{equation}
 with real parameters $C,\gamma,\delta,\delta_\nu$ satisfying
\begin{equation}\label{eq:totposcond}
   C>0,\quad \gamma\ge 0,\quad 0<\gamma+\sum_{\nu=1}^\infty \delta_\nu^2<\infty.
\end{equation}
\end{tm}

A comprehensive study of total positivity is given in the book of
Karlin \cite{Karlin1968a}. It is known that, if  $g$ is totally positive and integrable,
then $g$ decays exponentially (see \cite[p.~340]{Schoenb1951a}).
In this article, we restrict our attention to the class of totally positive functions
$g\in L^1(\bR)$ with the factorization
\begin{equation}\label{eq:totposI}
   \Phi(s)=\frac{1}{\cL[g](s)} =
   Ce^{\delta s}\prod_{\nu=1}^{M}(1+\delta_\nu s) \, ,
\end{equation}
for $M\in \bN $ and real $\delta _\nu $.
This means that   the denominator of $\cL[g]$ has only finitely many roots.
Equivalently, the Fourier transform of $g$ can be extended to a
meromorphic function with a finite number of poles on the imaginary
axis and no other poles.
As noted in \cite[p.~247]{SchoenbWhit1953a}, the exponential factor
can be omitted, as it corresponds to a simple
shift of $g$. In the following we will call a totally
positive function satisfying \eqref{eq:totposI} \emph{totally
  positive of finite type} and refer to $M$ as the type of $g$.

Schoenberg and Whitney  \cite{SchoenbWhit1953a} gave a complete characterization
of the case when the determinant $D$ in \eqref{eq:Dpos}  satisfies $D>0$.

\begin{tm}[\cite{SchoenbWhit1953a}] \label{strictpos}  Let $g\in L^1(\bR)$ be  a
  totally-positive function of finite type. 
Furthermore, let $m$ be the number of positive $\delta_\nu$ and
$n$ be the number of negative $\delta_\nu$ in \eqref{eq:totposI}, 
and $m+n\ge 2$.
For a set of points in
\eqref{eq:numbers},
the determinant $D = \det [ g(x_j - y_k)]_{j,k = 1, \dots , N}$ is strictly positive,
if and only if
\begin{equation}\label{eq:SWnumbers}
   x_{j-m}< y_j<x_{j+n}\qquad \mbox{for}~~ 1\le j\le N.
\end{equation}
Here, we use the convention that $x_j=-\infty$, if $j<1$, and
$x_j=\infty$, if $j>N$.
\end{tm}

The conditions in \eqref{eq:SWnumbers} are nowadays called the Schoenberg-Whitney
conditions for $g$. They have been used extensively in the analysis of spline
interpolation by Schoenberg and others (see the monograph~\cite{Sch81}).  They will  be crucial  for our
construction of a left inverse of the pre-Gramian matrix in \eqref{eq:1}.

As a generalization of the pre-Gramians  in \eqref{eq:1},
we consider bi-infinite matrices of the form
\begin{equation}\label{eq:biinf}
    P= \left[ g(x_j-y_k)\right]_{j,k\in \bZ} ,
\end{equation}
where each  sequence $X=(x_j)_{j\in\bZ}$ and $Y=(y_k)_{k\in\bZ}\subseteq
\bR $ is
strictly increasing.
Moreover,  the sequence
 $(x_j)_{j\in\bZ}$ is supposed to be denser
in the  sense of the following condition:
\[
  (C_r)~~\left\{
  \parbox{0.8\textwidth}{\parindent0pt
   (a)   ~~ every  interval $(y_k,y_{k+1})$ contains at least one point $x_j$;\\[5pt]
   (b)   ~~ there is an $r\in \bN$ such that  $|(y_k,y_{k+r})\cap X|\ge r+1$
   for all $k$.
   }\right.
\]

Our main tool for the study of Gabor frames and sampling theorems
will be the following technical  result. It can be interpreted as a
suitable extension of total positivity to infinite matrices.

\begin{tm}\label{main0} Let $g\in L^1(\bR)$ be  a  totally-positive
  function of finite type. 
 Let $m$ be the number of positive $\delta_\nu$,
$n$ be the number of negative $\delta_\nu$ in \eqref{eq:totposI}, and
$M=m+n\ge 1$.
Assume that the sequences  $(x_j)_{j\in\bZ}$ and
$(y_k)_{k\in\bZ}\subset \bR $  satisfy condition ($C_r$).

Then the matrix $P=\left[ g(x_j-y_k)\right]_{j,k\in \bZ}$
defines a bounded operator on $\ell_2(\bZ)$. It has an algebraic
left-inverse $\Gamma=\left[ \gamma_{k,j}\right]_{k,j\in \bZ}$,
and
\begin{equation}
  \label{eq:ch24}
    \gamma_{k,j}=0,\qquad\text{if}\quad x_j<y_{k-rm}~~
    \text{or}~~x_j>y_{k+rn}.
\end{equation}
  \end{tm}

\begin{proof}
We construct a left-inverse $\Gamma$ with the desired properties by defining each
row of $\Gamma$ separately.
It suffices to consider the row  with index $k=0$, as the construction of all
other rows is done in the same way.  The goal of  the first three  steps is
to select a  finite subset of  $x_j$'s and $y_k$'s that satisfy the
Schoenberg-Whitney conditions. (Our choice of indices is not unique and
not symmetric in $m$ and $n$, it minimizes the number of case
distinctions.)

\textbf{Step~1: Column selection.}
First, consider the case $m,n>0$ and set
$N:=(m+n-1)(r+1)$, if $n>1$,
and $N=m(r+1)+1$, if $n=1$.
We define an $N\times N$  submatrix $P_0$ of $P$ in the following way.
As columns of $P_0$, we select columns of $P$ between
\[
   k_1=-(r+1)m+1\quad\text{ and  }\quad k_2=k_1+N-1 \, .
\]
 Hence $k_2=(r+1)(n-1)$ for  $n>1$, and
 $k_2=1$ for  $n=1$. For later purposes, note that $k_1\le -m<0<n\le k_2$. Therefore, the
 column with index $k=0$ has at least $m$ columns to its left and $n$ columns to its right.

\textbf{Step~2: Selection of a square matrix. }
 Assumption ($C_r$) and our definition of $N$ imply that
the interval $I=(y_{k_1+m-1},y_{k_2-n+1})$ contains at least
$N$ points $x_j$. More precisely,  for $n>1$ we write
\[
   (y_{k_1+m-1},y_{k_2-n+1})=(y_{-rm},y_{r(n-1)}) = \bigcup _{\nu =
  -m}^{n-2} (y_{r\nu},y_{r(\nu+1)})
\]
and find at least
$r+1$ points $x_j$ in each subinterval
$(y_{r\nu},y_{r(\nu+1)})$ with $-m\le \nu\le n-2$. This  amounts to at
least $(m+n-1)(r+1) = N$ points in $I$.  If  $n=1$,
we have $(y_{k_1+m-1},y_{k_2-n+1})=(y_{-rm},y_1)$ and find
$m(r+1)$ points $x_j$ in
$(y_{-rm},y_{0})$ plus  at least  one additional point in $(y_0,y_1)$.

We let
\begin{equation}
  \label{eq:ch23}
   j_1:= \min\{j: x_j>y_{k_1+m-1}\},\qquad      j_2:= \max\{j:
   x_j<y_{k_2-n+1}\}\, .
\end{equation}
  We have just shown that  the set
\[
    X_0=\{x_j: j_1\le j\le j_2\}\subset (y_{k_1+m-1},y_{k_2-n+1})
\]
contains at least $ N$ elements. We choose a subset
\[
   X_0'=\{\xi_1<\cdots<\xi_N\}\subset X_0,
\]
that contains  precisely $N$ elements and satisfies
\[
   (y_k,y_{k+1})\cap X_0'\ne \emptyset\quad\text{ for }\quad
   k_1+m-1\le k\le k_2-n.
\]
That is, we choose one point $x_j$ in  each
interval $(y_k,y_{k+1})$,
with $k_1+m-1\le k\le k_2-n$,
and an additional $n+m-1$ points $x_j\in (y_{k_1+m-1},y_{k_2-n+1})$.
Note that
\begin{equation}\label{eq:X0prime}
  y_{k_1+m-1}<\xi_1=\min X_0' < y_{k_1+m} <
   y_{k_2-n}<\max X_0'=\xi_N < y_{k_2-n+1}.
\end{equation}
Now set
\[
   \eta_k= y_{k_1+k-1},\qquad 1\le k \le N \, ,
\]
and define  the matrix
\[
   P_0=(g(\xi_j-\eta_k))_{j,k=1,\ldots,N}.
\]
Then $P_0$ is a quadratic $N\times N$-submatrix of $P$.

\textbf{Step~3: Verification of the Schoenberg-Whitney conditions.} We
next show that $P_0$ is invertible  by checking the
Schoenberg-Whitney conditions. First, by \eqref{eq:X0prime}, we have
\[
   \xi_1=\min X_0' < y_{k_1+m} = \eta_{m+1}.
\]
By the construction of $X_0'$, this inequality progresses from left to right, i.e.,
\[
   \xi_j< y_{k_1 + m -1 +j} = \eta_{j+m}\qquad\text{for}\quad 1\le j\le N-m.
\]
Likewise,  we also have
\[
   \eta_{N-n}= y_{k_2-n} <\max X_0'=\xi_N,
\]
and  this inequality progresses from right to left, i.e.,
\[
   \eta_j<\xi_{j+n}\qquad\text{for}\quad 1\le j\le N-n.
\]
Therefore, the Schoenberg-Whitney conditions~\eqref{eq:SWnumbers}  are satisfied,
and  Theorem~\ref{strictpos}  implies that $\det P_0>0$.

\medskip
\textbf{Step~4: Linear dependence of the remaining  columns of $P$.}
We now make some important  observations.

 Choose indices $k_0$ and $s\in \bZ$  with $k_0< k_1$ and $m<s \le N$,
 and consider the new
  set  $\{ \eta _k ' : k=1, \dots , N\} $ consisting of the points $$y_{k_0}
< y_{k_1} <  \cdots < y_{k_1 +s -2} <y_{k_1+s}< \cdots  < y_{k_2}$$
  and  the corresponding $N\times N$-matrix $P_0' = ( g(\xi
_j - \eta _k ')) _{j,k = 1, \dots , N}$. This matrix is
obtained from $P_0$ by
 adding the column $\big(g(\xi_j-y_{k_0})\big)_{1\le j\le N}$ as its first
 column and deleting the column  $\big(g(\xi_j-\eta _s)\big)_{1\le
   j\le N}$. Then  $\eta_{m}$ appears in the $m+1$-st column of
 $P_0'$. By \eqref{eq:X0prime}, we see that
\[
   \eta _{m+1}' = \eta _m  = y_{k_1 +m-1} <  \xi _1.
\]
Consequently, the Schoenberg-Whitney conditions are violated
and therefore $\det P_0' = 0$ by Theorem~\ref{strictpos}. Since this holds for all $s >m$,
the vector   $(g(\xi_j-y_{k_0}))_{1\le j\le N}$ must be  in the linear
  span of the first $m$ columns of $P_0$, namely
  $(g(\xi_j-y_k))_{1\le j\le N}$ for $k= k_1, , \dots k_1 + m-1$.

 Likewise, choose  $k_3>k_2$, $1\le s \le N-n$, and  consider the new
  set  $\{ \eta _k'' : k=1, \dots , N\} $ consisting of the points $$
 y_{k_1} < \cdots <  y_{k_1 + s -2} <  y_{k_1 + s} <\cdots <
 y_{k_2}   < y_{k_3}$$ and  the corresponding $N\times N$-matrix
 $P_0'' = ( g(\xi
_j - \eta _k '')) _{j,k = 1, \dots , N}$. This matrix is
obtained from $P_0$ by
 adding the column $(g(\xi_j-y_{k_3}))_{1\le j\le N}$ as its last (=$N$-th)
 column and deleting the column  $\big(g(\xi_j-\eta _s)\big)_{1\le
   j\le N}$. Then  $\eta_{N-n+1}$ appears in the $n+1$-st column of
 $P_0''$, counted from  right to left,
 and $$\eta _{N-n}'' = \eta _{N-n+1}   = y_{k_1 +N-n} = y_{k_2 -n+1 } >
   \xi _N $$ by \eqref{eq:X0prime}. Again the Schoenberg-Whitney conditions are violated and
therefore $\det P_0'' = 0$. We conclude that the vector
$(g(\xi_j-y_{k_3}))_{1\le j\le N}$ must lie  in the linear
  span of the last $n$ columns of $P_0$.


\textbf{Step~5: Construction of the left-inverse. }
Recall that $k_1 = -m(r+1)+1$ and that $\eta _{(r+1)m} = y_{k_1 +
  (r+1)m-1} = y_0$.
Let $ c^T$ denote
the $(r+1)m$-th  row vector of $P_0^{-1}$.
 By definition of the
inverse, we have
$\sum _{j=1}^N c_j g(\xi _j - \eta _k) = \delta _{k , (r+1)m}$,
or equivalently, for $k_1 \leq k \leq k_2$,
\begin{equation}\label{eq:ckg}
\sum _{j=1}^N c_j g(\xi _j - y_k ) = \delta _{k , 0}  \, .
\end{equation}
Let us now consider the other columns with $k<k_1$ or $k>k_2$.
Since $k_1\le -m<0<n\le k_2$ and every vector $(g(\xi_j-y_k))_{1\le
  j\le N}$ lies either in the span of the first $m$ columns of $P_0$
(for $k< k_1$) or in the span of the last $n$ columns of $P_0$ (for
$k>k_2$), we obtain that
$$
\sum _{j=1}^N c_j g(\xi _j - y_k ) = 0\, .
$$
Therefore, the identity \eqref{eq:ckg} holds for all $k\in \bZ$.

Next we fill the vector $c$ with zeros and define the infinite vector
$\gamma _0$ by
$$
\gamma _{0,j} =
\begin{cases}
  c_{j'} & \text{ if } x_j=\xi_{j'} \in X_0' \\
 0 & \text{ otherwise}.
\end{cases}
$$
Then
$$
\sum _{j\in \bZ } \gamma _{0,j} g(x_j - y_k ) = \sum _{j=1}^N
c_{j} g(\xi _j - y_k ) = \delta _{k,0} \, .
$$
Thus $\gamma _0 $ is a row of the left-inverse of $P$. By
construction, $\gamma _0$ has at most $N$ non-zero entries.
In particular,  if $x_j<y_{k_1+m-1} = y_{-rm}$, then  we have
$j<j_1$ and thus  $\gamma_{0,j}=0$.
Similarly,  if $x_j>y_{k_2-n+1} = y_{r(n-1)}$ for $n>1$,
and $x_j>y_{k_2-n+1} =y_1$ for $n=1$,
then we have
$j>j_2$ and  $\gamma_{0,j}=0$.
This gives the support  properties of the entries $\gamma_{0,k}$
of row $k=0$ of the left-inverse $\Gamma$.



\bigskip

\textbf{Step~6: The other rows of $\Gamma $.} The construction of the
$k$-th row of $\Gamma $ is similar. We choose the columns between $k_1
= k - (r+1)m +1$ and $k_2 = k_1 +N-1$ of $P$, and, accordingly, we
choose suitable rows between the indices
$j_1 = \min \{j: x_j>y_{k_1+m-1}\}$ and $j_2:= \max\{j:
   x_j<y_{k_2-n+1}\} $. Then the column  of $P$ containing $y_k$ has at
   least $m$ columns to its left and $n$ columns to its right. We then
   proceed to select $\xi _j$'s and define an $N\times N$-matrix $P_k $ and
   verify that $\det P_k >0$. The $k$-th row of $\Gamma $ is
   obtained by   padding  the appropriate
   row (with row-index $(r+1)m$) of $P_k^{-1}$ with zeros.
   By this construction one obtains a vector $\gamma_k=
   (\gamma_{k,j})_{j\in\bZ}$,
   for which
$$
\sum _{j\in \bZ } \gamma _{k,j} g(x_j - y_l ) =  \delta _{k,l}
$$
holds. Furthermore, $\gamma_{k,j}=0$  when $ x_j<y_{k-rm}$ and when $x_j>y_{k+rn}$.

\bigskip

\textbf{Step~7: The remaining cases. }
The cases where $m=0$ or $n=0$ are simple adaptations of the above steps. For $m=0$,
$n\geq 2$  we choose $N=(n-1)(r+1)$. The indices for the submatrix
$P_k$ that occurs in the construction of the $k$-th row of $\Gamma$
are $k_1=k$, $k_2= k+N-1$, $j_1 = j_1(k) = j_1=\max\{j: x_j<y_{k_1}\}$
and $j_2= j_2(k) = \max \{ j: x_j < y_{k_2}\}$. Now we proceed as
before.

We note that Step~4 simplifies a bit. For  $m=0$  the
function $g$ is supported on $(-\infty , 0)$
by~\cite[p.~339]{Schoenb1951a}. Consequently,  
 if $k_0 < k_1$ and $j\geq j_1$, then $x_j - y_{k_0} 
>0$ and  $g(x_j -
y_{k_0}) = 0$.

Thus  the column vectors of $P$ to the left of the submatrix $P_0$
are identically zero and no further proof is needed for linear dependence.

 The case $n=0$ is similar. It can be reduced to the previous case by
a reflection $x\to -x$, which interchanges the role of $m$ and $n$. 

 The special case of $m=0$ and
$n=1$ can be solved by taking $N=2$, $k_1 = k$,  $k_2=k+1$, $j_1=\max
\{ j: x_j < y_k\}$,
$j_2=\min\{j: x_j > y_k\}$, and the $2\times 2$-matrix
\[
    P_k = \left(\begin{matrix} g(x_{j_1}-y_k) & g(x_{j_1}-y_{k+1}) \\
    g(x_{j_2}-y_k)& g(x_{j_2}-y_{k+1})\end{matrix}\right).
\]
In this simple case the size of the matrix is independent of the
parameter $r$ occurring in condition $C_r$.
\end{proof}

\section{Gabor Frames with Totally Positive functions} \label{3}

In this section we prove the main result about Gabor frames.
Recall that a totally positive function is said to be of finite type,
if its two-sided Laplace transform factors as $\cL [g](s)\inv  =  Ce^{\delta
  s}\prod_{\nu=1}^{M}(1+\delta_\nu s)$ with real numbers $\delta,\delta_\nu$.

\begin{tm}\label{main1}
Assume that $g$ is a totally positive function of finite type and
$M=m+n\geq 2$, where $m$ is the number of positive zeros
and $n$ the number of negative zeros of $1/\cL(g)$.
Then $\cF (g) =  \{ (\alpha , \beta ) \in \bR _+^2 : \alpha
\beta <1\}$.
The Gabor frame $\gabsy $ possesses a piecewise
continuous dual window $\gamma$ with compact support
in $[-\tfrac{rm}{\beta}-\alpha,
\tfrac{rn}{\beta}+\alpha]$,
where $r:=\lfloor \tfrac{1}{1-\alpha\beta} \rfloor$.
\end{tm}

By taking a \ft , we obtain the following corollary.

\begin{cor}
  If $h(\tau ) =  C \prod_{\nu=1}^{M}(1+2\pi i \delta_\nu \tau )\inv
  $ for $M\geq 2$, then $\cF (h) =  \{ (\alpha , \beta ) \in \bR _+^2 : \alpha
\beta <1\}$ and $\cG (h, \alpha, \beta )$ possesses a bandlimited dual
window $\theta$ with $\supp \, \hat{\theta} \subseteq [-\tfrac{rm}{\alpha}-\beta,
\tfrac{rn}{\alpha}+\beta]$.
\end{cor}

\rem\ We have excluded the case  $M=1$ for the formulation of the
theorem, because it  corresponds to the one-sided exponential $g(t) =
e^{-t} \chi _{\bR _+ }(t)$. This function is discontinuous and
therefore is not subject to the Balian-Low principle. Instead, we
have only $\cF (g) =  \{ (\alpha , \beta ) \in \bR _+^2 : \alpha
\beta \leq 1\}$.

\begin{proof}[Proof of Theorem~\ref{main1}]
  Since $g\in L^1(\bR )$ is totally positive, it decays
  exponentially, and since $M\geq 2$, its Fourier transform
  $\hat{g}(\xi )= C e^{2\pi i \delta \xi}
  \prod_{\nu=1}^{M}(1+2\pi i \delta_\nu \xi )\inv $
  decays at least like $ |\hat{g}(\xi )| \leq \tilde C(1+\xi ^2)\inv $. In
  particular, $g$ is continuous.  As a
  consequence, the assumptions of the Balian-Low theorem are
  satisfied~\cite{BHW95,book} and $\cF (g) \subseteq \{ (\alpha , \beta ) \in \bR _+^2 : \alpha
\beta < 1\}$.

To  prove that $\gabsy $ is a frame for $\alpha \beta <1$, we will
construct  a family of uniformly bounded  left-inverses for the
pre-Gramians $P(x)$ of $g$ and then use
Lemma~\ref{l:pre}.

Fix  $x\in [0,\alpha ]$ and consider  the sequences
$x_j = x+\alpha j$ and $y_k = k/\beta $, $j,k\in \bZ $. We first check condition
$(C_r)$.
By our assumption, we have $\alpha  <1/\beta$
and every interval $(k/\beta , (k+1)/\beta )$
contains at least one point $x+\alpha j$. Every interval   $(k/\beta ,
(k+r)/\beta )$, with $r\in\bN$,
 contains at least $r+1$ points $x+\alpha j$, if
$r/\beta > (r+1)\alpha$, i.e., we have
$$
r > \frac{\alpha \beta }{1-\alpha\beta},\quad\text{or equivalently}\quad
r\ge \left\lfloor \frac{1}{1-\alpha\beta} \right\rfloor\, .
$$
Consequently, condition $(C_r)$ is satisfied with $r = \lfloor \frac{1 }{1-\alpha\beta} \rfloor $.
 By  Theorem~\ref{main0}, each  pre-Gramian
$P(x)$ with entries $g(x+\alpha j -k/\beta )$
possesses a left-inverse $\Gamma (x)$.

To apply Lemma~\ref{l:pre1}, we need to show that $\Gamma (x), x\in [0,\alpha ]$, is a uniformly
bounded set of operators on $\ell ^2(\bZ )$.

Let $P_0(x)$ be the $N\times N$-square submatrix constructed in
Steps~1 and~2. The column indices $k_1 $ and $k_2$ depend only on the
type of $g$, but not on $x$. The row indices $j_1 = j_1(x) = \min \{ j
:  x_j  > y_{k_1  +m-1} \} = \min \{ j :    x+\alpha j> (k_1  +m-1)/\beta
  \}$ and $j_2 = j_2(x)$ are locally constant in $x$. Likewise the
  indices that determine which rows $j, j_1 < j < j_2$ of $P(x)$ are
  contained in $P_0(x)$ are locally constant. Consequently, for every
  $x\in [0,\alpha ]$ there is a neighborhood $U_x$ such that indices
  used for   $P_0(y)_{jk } = g(y+\xi _j - k/\beta ) $ do not depend on
  $y\in U_x$. Since $g$ is continuous, $P_0(y)$ is continuous on
  $U_x$, and since $\det P_0(x) >0$ there exists a neighborhood $V_x
  \subseteq U_x$, such that $\det P_0(y) \geq \det P_0(x)/2$ for $y
  \in V_x$.

We now cover $[0,\alpha ]$ with finitely many
neighborhoods $V_{x_q}$ and obtain that $\det P_0(y) \geq \min _q \det
P_0(x_q)/2 = \delta >0$ for all $y\in [0,\alpha ]$. Since   each entry
of the inverse  matrix $P_0(y)\inv$  can be calculated by Cramer's
rule, these entries must be  bounded by $C \det P_0(y)\inv \leq
C\delta \inv  $ with a constant $C$
depending only on $\|g\|_\infty $ and the dimension $N$ of
$P_0(y)$.


By construction (Step~5), the zero-th row $\gamma (x)=(\gamma_{0,j}(x))_{j\in\bZ} $
 of the left-inverse $\Gamma (x) $
contains at most $N\leq (r+1)M $ non-zero entries, namely those of  the row
of  $P_0(x)\inv $ corresponding to $y_0 = 0$.  We have thus
constructed a vector-valued function $x\to \gamma (x)$ from
$[0,\alpha ] \to \ell ^\infty (\bZ )$ with the following properties:
\begin{itemize}
\item[(i)] $\gamma (x)$ is piecewise continuous,

\item[(ii)]  $\mathrm{card}\, (\supp \, \gamma (x) ) \leq (r+1)M$,
  where according to~\eqref{eq:ch24}
$$\supp \,
\gamma (x) =\{j: \gamma _{0,j}(x)\not\equiv 0\}
 \subseteq \{ j : \frac{-rm}{\beta} \leq x+\alpha j \leq
\frac{rn}{\beta }\} \subseteq \{ j : \frac{-rm}{\alpha \beta}-1 \leq  j \leq
\frac{rn}{\alpha \beta }\} \, ,
$$

\item[(iii)]  and  $\sup _{x\in [0,\alpha ]}  \|\gamma (x) \|_\infty =
C <\infty $.
\end{itemize}

Consequently, the  dual window $ \gamma (x) =  \beta  \sum _{j\in
  \zd } \overline{\gamma } _{0,j}(x) \chi _{[0,\alpha
    )}(x-\alpha j) $ corresponding to $\Gamma (x)$  by
Lemma~\ref{l:pre1},    has compact support on the interval
$[-\tfrac{rm}{\beta}-\alpha,\tfrac{rn}{\beta}+\alpha] $,
is piecewise continuous,  and
is bounded. In particular, it satisfies the Bessel property
(see~\cite{walnut92} or \cite[Cor.~6.2.3]{book}).

We have constructed a dual window for $\gabsy $ satisfying the Bessel
property. By Theorem~\ref{charone} and Lemma~\ref{l:pre},
 $\gabsy $ is a Gabor frame.
\end{proof}

\subsection{Remarks and Conjectures}

In the proof of Theorem~\ref{main1} we have constructed  a compactly
supported dual window $\gamma $ for  $\gabsy $. This construction is
explicit and can be realized numerically, because it requires only the
inversion of finite matrices.  To determine the
values $\gamma (x+j\alpha )$, one has to solve the  linear $N\times N$
system $P_0(x) \gamma (x) = e$ for a  vector $e $ of the standard
basis of $ \bR ^N$.

We observe that the canonical dual window (provided by standard frame
theory) has better smoothness properties. The regularity theory for
the Gabor frame operator implies that the canonical dual window
$\gamma ^\circ $ decays exponentially and its Fourier transform
$\widehat{\gamma ^\circ}(\xi )$ decays like $\cO ( |\xi | ^{-M })$,
where $M$ is the type of $g$. See~\cite{delprete,book,GL04,strohmer00}.

\vs

Theorem~\ref{main1} raises many new questions. Theorem \ref{main1}
suggests the  natural
conjecture that the frame set of  \emph{every }
totally positive continuous function $g$  in $L^1(
\bR )$   is $\cF (g) = \{ (\alpha , \beta ) \in \bR _+^2 :
\alpha \beta < 1\}$. Our proof is tailored to totally positive
functions of finite type, most likely the proof of the conjecture will require a
different method.

In a larger context one may speculate about the set $\cM$ of functions such
that the frame set is exactly $\cF (g) = \{ (\alpha , \beta ) \in \bR _+^2 :
\alpha \beta < 1\}$. In other words, when is the necessary density
condition $\alpha \beta <1$  also sufficient  for $\cG (g,\alpha ,
\beta )$ to be a frame? The invariance properties of Gabor frames imply
that the class  $\cM$ must be invariant under \tfs s, dilations,
involution,  and the
\ft . Furthermore, if $\cF (g) = \{ (\alpha , \beta ) \in \bR _+^2 :
\alpha \beta < 1\}$, then both $g$ and $\hat{g}$ must have infinite
support.

A general method for constructing functions in $\cM$  can be extracted
from ~\cite{JS02}.  We write $\hat{c}(\xi ) = \sum _{k\in \bZ } c_k
e^{2\pi i k\xi }$ for the Fourier series of a sequence $(c_k)$ and
then define, for a given function $g_0\in L^2(\bR )$,
$$
C_{g_0} = \{ f \in L^2(\bR ): f = \sum _{k,l\in \bZ } c_k d_l T_k M_l
g_0, c, d \in \ell ^1(\bZ ),  \inf _{\xi } (|\hat{c}(\xi ) \,
\hat{d}(\xi )| >0  \} \, \, .
$$

\begin{lemma}\label{trick}
  Let $g_0$ be a totally positive function of finite type $M\geq
  2$. If $g \in C_{g_0}$, then the frame set of $g$ is  $\cF (g) = \{ (\alpha , \beta ) \in \bR _+^2 :
\alpha \beta < 1\}$.
\end{lemma}
This trick was used by Janssen and Strohmer in\cite{JS02}.  They
showed that the hyperbolic secant $g_1(t)=  (e^t+e^{-t})\inv$ belongs to the
set $C_{\vf }$ for the Gaussian window $\vf (t) = e^{-\pi t^2}$ and
then  concluded that $\cF (g_1) = \{ (\alpha , \beta ) \in \bR _+^2 :
\alpha \beta < 1\}$. The general argument is identical.

Each class $C_g$ is completely determined by the zeros of the Zak
transform of $g$. Let $Zg(x,\xi ) = \sum _{k\in \bZ } g(t-k) e^{2\pi i
  k\xi }$ be the Zak transform of $g$. Since $Z(T_k M_l g)(x,\xi ) =
e^{2\pi i (lx-k\xi)} Zg(x,\xi )$, every $g\in C_{g_0}$ has a Zak transform of
the form
$$ Zg(x,\xi ) = \hat{c}(-\xi ) \hat{d}(x) Zg_{0}(x,\xi) \, .
$$
The definition of $C_{g_0}$ implies that $Zg$ and $Zg_{0}$ have the
same zero set. If $g_0$ and $h$ are two totally positive functions,
then the zero sets of $Zg_0$ and $Zh$ are different in general,
therefore Lemma~\ref{trick} leads to  distinct sets $C_{g_0}$.  The
zeros of the Zak transform seems to be some kind of invariant for the
Gabor frame problem, but  their
deeper significance  is still mysterious.

\section{Sampling Theorems}

In this section we exploit the connection between Gabor frames and
sampling theorems and   prove new and sharp  sampling theorems for shift-invariant
spaces. Originally shift-invariant spaces  were used as a substitute for bandlimited
functions and were defined as the span of integer translates of a given
 function $g$. We refer to the survey ~\cite{AG01} for the theory of  sampling in
 shift-invariant spaces.   We will deal with a slightly more
 general class of spaces that are generated by arbitrary shifts.

Let $Y=(y_k)_{k\in\bZ}$ be  a strictly increasing sequence and
consider the quasi shift-invariant space
$$
V_Y(g) = \{ f \in L^2(\bR ): f= \sum _{k\in \bZ } c_k g(.-y_k) \} \, .
$$

We require
that the sequence $Y=(y_k)$ of shift parameters  satisfies the
conditions
\begin{equation}\label{eq:qy}
  0< q_Y=\inf_k(y_{k+1}-y_k)\le \sup_k(y_{k+1}-y_k)=Q_Y<\infty.
\end{equation}
Such sequences are  called quasi-uniform or uniformly discrete. 
The numbers $Q_Y,q_Y>0$ are  the mesh-norm and the
separation distance of $Y$. 

For the norm equivalence $\|f\|_2 \asymp \|c\|_2$ for $f\in V_Y(g)$ we
need that $\{g(.-y_k: y_k \in Y\}$ is a Riesz basis for $V_Y(g)$.

\begin{lemma} \label{rbsi}
Let  $g$ be an arbitrary totally positive function.

(i) If   $Y=h \bZ$ with $h>0$, then
$\{ g( .  - h k ), k\in \bZ \}$ is a Riesz basis for $V_Y(g)$.

(ii) If $Y$ is quasi-uniform, then $\{g (. - y_k) : k\in \bZ \}$ is a Riesz basis for $V_Y(g)$.
\end{lemma}

\begin{proof}
(i)   Since $\hat{g}$ is continuous, does not have any
  real zeros, and $\hat{g}(\xi)$ decays at least like $C/|\xi|$, 
  every periodization of
  $|\hat{g}|^2$ is bounded above and below. This property is
  equivalent to the Riesz basis property, e.g.~\cite[Thm.~7.2.3]{chr03}.

Of course, (i) also follows from (ii).

(ii) For the general case we use  Zygmund's
inequality~\cite[Thm.~9.1]{zygmund}: If $I$ is an interval of length
$|I|> \tfrac{1+\delta }{q_Y}$, then
$$
\int _{I} \big|\sum _{k} c_k e^{2\pi i y_k \xi }\big|^2 \, d\xi \geq A_\delta
|I| \|c\|_2^2
$$
for a constant depending only on $\delta >0$.

If $ f = \sum _k c_k g(.- y_k)$, then
\begin{align*}
  \|f\|_2^2 &= \|\hat{f}\|_2^2 = \int _{\bR } \big|\sum _k  c_k e^{-2\pi i
    y_k \tau } \big|^2 \, |\hat{g} (\tau) |^2\, d\tau \\
&\geq \inf_{\tau\in I} |\hat{g} (\tau) |^2 \, \int _I   \big|\sum _k  c_k e^{-2\pi i
    y_k \tau } \big|^2 \, d\xi \\
&\geq C |I| A_\delta \|c\|_2^2 \, .
\end{align*}
Here $\inf_{\tau\in I} |\hat{g} (\tau) |^2 >0$, because $\hat{g}$ does not have
any real zeros by Theorem~\ref{thm:schoen}.
\end{proof}

We are interested to derive sampling theorems for generalized
shift-invariant spaces that are generated by a totally positive
function $g$.  Our goal is  to
construct  strictly increasing sequences  $X = (x_j)$ that yield a  sampling
inequality
\begin{equation}
  \label{eq:c89}
  A \|f \|_2 ^2 \leq \sum _{j\in \bZ } |f(x_j)|^2 \leq B \|f\|_2^2
  \qquad \text{ for all } f\in V_Y(g) \,
\end{equation}
for some  constants $A,B >0$ independent of $f$. Following
Landau~\cite{landau67},   a set $X \subset
\bR $ that  satisfies the norm equivalence \eqref{eq:c89}  is called a
set of (stable) sampling for $V_Y(g)$. Except for bandlimited
functions and  B-spline generators only qualitative results are known
about sets of sampling in shift-invariant spaces.

We first give an equivalent condition for sets of sampling in
$V_Y(g)$.
As in Lemma~\ref{l:pre} we obtain the following characterization of
sets of sampling in $V_Y(g)$.

\begin{lemma} \label{chartwo}
Let $g\in L^2(\bR)$, and let $Y=(y_k)_{k\in \bZ}\subset \bR$ be a strictly increasing sequence.
Then a set $\{x_j\}\subset \bR$ is a set of
sampling for $V_Y(g)$, \fif\ the pre-Gramian $P$ with entries $p_{jk}=g(x_j -
y_k )$ possesses a left-inverse $\Gamma  $ that is bounded on
$\ell ^2 (\bZ )$.
\end{lemma}

The case of uniform sampling in shift-invariant spaces is completely
settled by the results in Section~\ref{3}.

\begin{cor}
  Let $g$ be a totally positive function of finite type $M\geq 2$ and
  $Y= h \bZ $. If $\alpha  < h $ and $x\in \bR $ is
  arbitrary,  then the set $x+\alpha \bZ $ is a set of sampling for
  $V_{Y}(g)$.  More precisely, there exist positive constants $A,B$
  independent of $x$, such that
$$
  A \|f \|_2 ^2 \leq \sum _{j\in \bZ } |f(x +\alpha j)|^2 \leq B \|f\|_2^2
  \qquad \text{ for all } f\in V_Y(g) \, .
$$
\end{cor}
\begin{proof}
We proved  Theorem ~\ref{main1} by verifying the equivalent condition
of Theorem~\ref{charone}, namely \eqref{eq:4} stating that
\begin{equation}
  \label{eq:4a}
A\|c\|_2^2 \leq \sum _{j\in \zd } \big| \sum _{k\in \zd } c_k
g(x+j\alpha - h k )\big|^2 \leq B \|c\|_2^2  \quad
\text{ for all } \,\, x\in \bR , c\in \ell ^2(\bZ  ) \, .
\end{equation}
Since $f\in V_Y(g)$ is of the form $f = \sum _{k} c_k g(. - h k )$
and $\|f\|_2 \asymp \|c\|_2$ by Lemma~\ref{rbsi}, the inequalities \eqref{eq:4a} are
equivalent to the sampling inequality $\|f\|_2^2 \asymp \sum _{j\in
  \bZ } |f(x+\alpha j)|^2 $, and the constants are independent of $x$
by Theorem~\ref{charone}.
\end{proof}

\rem\ The condition $\alpha  <h $ is sharp. If $\alpha = h$,
then there exists $x\in \bR $, such that $x+\alpha \bZ $ is not a set
of sampling. This follows immediately from the Balian-Low theorem
~\cite{BHW95}.

\medskip

Our methods yield more general sampling theorems. On the one hand, we
study non-uniform sampling sets, and on the other hand, we may treat
quasi shift-invariant spaces.
The auxiliary characterization of Lemma~\ref{chartwo} gives a hint of
how to proceed.
If the sequences $(x_j)$ and $(y_k) $ satisfy condition
$(C_r)$ for some $r>0$, then by Theorem~\ref{main0} the pre-Gramian matrix $P$
possesses an algebraic left-inverse. To obtain  a sampling theorem, we
need to impose additional conditions on $(x_j)$ and $(y_k)$, so that this
left-inverse is  bounded on $\ell ^2$.

\bigskip

To verify the boundedness of a matrix, we will apply the following
lemma which is a direct consequence of  Schur's test,
see, e.g.,  ~\cite[Lemma~6.2.1]{book}.

\begin{lemma}\label{lem:matrix}
    Assume that $\mathbf{A}=(a_{jk})_{j,k\in\bZ}$ is a matrix with bounded
    entries $|a_{jk}| \leq C $ for $j,k\in\bZ$. Furthermore, assume that
there exists a strictly increasing sequence $(j_k)_{k\in\bZ}$
of row indices $j_k\in\bZ$  and $N\in \bN$ such that    $a_{jk} = 0$
for $|j-j_k| \ge N$.   Then
  \begin{equation}
    \label{eq:som1}
    \|\mathbf{A}\|_{\ell ^2 \to \ell ^2  } \leq (2N-1)C \,.
  \end{equation}
\end{lemma}

\begin{proof} The conditions give
\[
   K_2:= \sup_{k\in\bZ} \sum_{j\in\bZ} |a_{jk}| =
   \sup_{k\in\bZ} \sum_{j=j_k-N+1}^{j_k+N-1} |a_{jk}|
   \le (2N-1)C.
\]
For the estimate of the column sums, 
we define the set
\[
   N_j=\{ k\in\bZ :  a_{jk}\ne 0\} \subseteq \{k\in\bZ :
   |j-j_k| < N\} \qquad \text{ for } j\in \bZ \, .
\]
Since $(j_k)$ is strictly increasing, $N_j$ has at most $(2N-1)$ elements, and
this gives
\[
   K_1:= \sup_{j\in\bZ} \sum_{k\in\bZ} |a_{jk}| =
   \sup_{j\in\bZ} \sum_{k\in N_j} |a_{jk}|
   \le (2N-1)C.
\]
The assertion now follows from Schur's test. 
\end{proof}

In the following we give a  sufficient condition for a set $X$ to be a
set of sampling for $V_Y(g)$.

\begin{tm} \label{sampsi}
  Let $g$ be a totally  positive  function of finite type $M\geq 2$. Let
  $Y=(y_k)_{k\in \bZ}\subset \bR$ be an increasing quasi-uniform
  sequence   with parameters $q_Y$, $Q_Y$ defined in \eqref{eq:qy}. Moreover, let $(x_j)_{j\in \bZ}\subset \bR$
  be a strictly increasing  sequence, which satisfies the following conditions:
\[
  (C_r(\epsilon))~~\left\{
  \parbox{0.8\textwidth}{\parindent0pt
  There exist $r\in \bN$, $\epsilon \in (0,q_Y/2)$, and
  a quasi-uniform subsequence $X'\subseteq X$,
  such that\\[5pt]
   (a)   ~~ every interval $(y_k+\epsilon,y_{k+1}-\epsilon)$ contains at
   least one point $x_j\in X'$;\\[5pt]
   (b)   ~~ for every $k\in \bZ$, we have $|(y_k+\epsilon,y_{k+r}-\epsilon)\cap X'|\ge r+1$.
   }\right.
\]
Then  $X$   is a set of
  sampling for $V_Y(g)$.
\end{tm}

\begin{proof}
\textbf{Step 1.} First,
we construct a left-inverse of the pre-Gramian $P$ 
as  in the proof of Theorem~\ref{main0}  with a small modification.

We  consider only  the case $m>1$, $n>1$. For the  construction of the row with
the index $k$ of the left-inverse $\Gamma$,  we choose the size  $N\in
\bN $ for a  square
submatrix  $P_k$  of $P$  as in Step 1 and  the column indices $k_1=
k-r(m+1)+1$ and $k_2= k_1 +N-1$ as in Step 6.

To incorporate condition $C_r(\epsilon )$, we  modify the selection of
the  row indices in Step 2 as follows.
 Assumption $C_r(\epsilon)$ and our definition of $N$ imply that
the interval $I=(y_{k_1+m-1}+\epsilon,y_{k_2-n+1}-\epsilon)$ contains at least
$N$ points $x_j\in X'$, where $X'$ is the quasi-uniform subset of $X$ in
 condition $C_r(\epsilon)$.
Define
\[
   j_1:= \min\{j: x_j\in X',~x_j\ge y_{k_1+m-1}+\epsilon\},
   \quad      j_2:= \max\{j:
   x_j\in X',~x_j\le y_{k_2-n+1}-\epsilon\}\, ,
\]
then the set
\[
    X_k=\{x_j\in X' : j_1\le j\le j_2\}\subset (y_{k_1+m-1}+\epsilon,y_{k_2-n+1}-\epsilon)
\]
has at least $ N$ elements. We now  choose  one point
$x_j\in (y_l+\epsilon,y_{l+1}-\epsilon)\cap X'$ for each
 $k_1+m-1\le l \le k_2-n$
and an additional $n+m-1$ points
$x_j\in (y_{k_1+m-1}+\epsilon,y_{k_2-n+1}-\epsilon)\cap X'$ and obtain
 a subset
\[
   X_k'=\{\xi_1<\cdots<\xi_N\}\subseteq X_k\subset X',
\]
containing precisely $N$ elements.
As before, we set
\[
   \eta_l= y_{k_1+l-1},\qquad 1\le l \le N \, ,
\]
and define  the quadratic  submatrix $P_k$ of $P$ by
\[
   P_k=(g(\xi_j-\eta_l))_{j,l=1,\ldots,N}\, .
\]
The modified  construction leads to a stronger version of the Schoenberg-Whitney
conditions, namely
\begin{equation}\label{eq:Xksample}
  \xi_j+\epsilon \le  \eta_{j+m}\text{ for }1\le j\le N-m,\qquad
  \eta_j+\epsilon \le  \xi_{j+n}\text{ for }1\le j\le N-n \, .
\end{equation}

Step 4 remains unchanged,  and the column of $P$ with $k<k_1$ or
$k>k_2$ are linearly dependent on the columns of $P_k$. 

Hence $P_k$ is invertible
and, by padding the $(r+1)m$-th row   of $P_k^{-1}$
 with zeros, we obtain the  row $(\gamma_{k,j})_{j\in\bZ}$
 with the row index $k$ of the left-inverse $\Gamma$.

\textbf{Step 2. } We  show that this left inverse $\Gamma $
defines a bounded operator on $\ell ^2(\bZ )$.

By construction the $j$-th row $(\gamma _{k,j})$ of $\Gamma $ has at most
$N$ non-zero entries between  $k_1 = k-(r+1)m+1$ and $k_1+N-1$.  To
apply Lemma~\ref{lem:matrix}, we  need to show that the entries of
$\Gamma $ are uniformly bounded, or equivalently,  that  the entries
of  $P_k^{-1}$ are  bounded with a bound that  does not depend on $k$.

We set up  a compactness argument similar to the proof of Theorem~\ref{main0}.

We begin with the simple observation that
\[
    g(\xi_j-\eta_l)= g\big( (\xi_j-\eta_1)-(\eta_l-\eta_1)\big),
    \quad j,l=1,\ldots,N\, .
\]
Let  $S$ be the $N$-dimensional simplex
\begin{equation}
    S = \{\tau=(\tau_1,\ldots,\tau_N) \in \bR ^N :
    ~ 0\le \tau_1 \le \cdots \le \tau_N\le (N-1) Q_Y\}\, .
\end{equation}
Although the  finite sequences $(\xi_j)_{1\le j\le N}$ and
$(\eta_l)_{1\le l\le N}$ depend on the row index $k$ (and we should
write $\xi _j^{(k)}$ and $\eta _l ^{(k)}$ to make the dependence
explicit), we always have
$$
0< \xi _1 - \eta _1 < \xi _N - \eta _1 < \eta _N - \eta _1 \leq
(N-1)Q_Y \, .
$$
Consequently,
\[
    (\xi_1-\eta_1,\ldots,\xi_N-\eta_1)\in S, \,\, \text{ and } \,\,
    (0,\eta_2-\eta_1,\ldots,\eta_N-\eta_1)\in S\, .
\]

Let $q:= \min\{q_{X'},q_Y\}>0$ be the minimum of the separation distances
of the quasi-uniform sets $X'$ and $Y$ and let
\begin{equation}
    S_q = \{\tau=(\tau_1,\ldots,\tau_N)\in S :  \tau_{j+1}-\tau_j \ge
    q\text{ for }1\le j\le N-1\}\, .
    \notag
\end{equation}
Then $S_q$ is  compact and
\[
    (\xi_1-\eta_1,\ldots,\xi_N-\eta_1)\in S_{q} \,\, \text{ and } \,\,
    (0,\eta_2-\eta_1,\ldots,\eta_N-\eta_1)\in S_q\, .
\]
Finally, we define the compact set
\[
   K=\{(\tau,\theta)\in S_q\times S_q : \,   \begin{array}[t]{l}
   \tau_j+\epsilon\le \theta_{j+m}~\text{ for }~1\le j\le N-m,\\[5pt]
   \theta_j+\epsilon\le \tau_{j+n}~\text{ for }~1\le j\le N-n~ \}\,.
   \end{array}
\]
The  assumption $C_r(\epsilon)$ implies that
 \[
    \Big((\xi_1-\eta_1,\ldots,\xi_N-\eta_1),
    (0,\eta_2-\eta_1,\ldots,\eta_N-\eta_1)\Big)\in K\, .
\]
Clearly, the Schoenberg-Whitney conditions are satisfied for every
point  $(\tau,\theta)\in K$ and therefore every  $N\times N$-matrix
$(g(\tau_j-\theta_l))$ has positive determinant. Since the determinant
depends continuously on $(\tau, \theta)$ and $K $ is compact,  we
conclude that
\[
    \inf_{(\tau,\theta)\in K} \det \Big(g(\tau_j-\theta_l)\Big) =
    \delta     >0.
\]
This construction implies that $\det P_k \geq \delta >0 $ for every
$k$.
As in the proof of  Theorem~\ref{main0} we use  Cramer's rule and
conclude  that
 all entries of $P_k^{-1}$
  are bounded by $ (N-1)! \, \delta^{-1} \|g\|_\infty^{N-1}$.

The assumptions of the modified Schur test are satisfied, and
Lemma~\ref{lem:matrix} yields that the  matrix $\Gamma$ is bounded
as an operator on $\ell^2(\bZ)$. Finally,  Lemma~\ref{chartwo}
implies that $X$ is a set of sampling for $V_Y(g)$.
\end{proof}

\begin{cor} \label{finalcor}
Assume that $g$ is totally positive of finite order $M\geq 2$ and
$Y= h \bZ$. Let  $\alpha =\sup _{j\in \bZ } (x_{j+1} - x_j)  $ be the maximum gap
between consecutive sampling points.
If $\alpha < h $, then
$(x_j)$ is a set of sampling for $V_Y(g)$.
\end{cor}

\begin{proof}
 The assumption of Theorem~\ref{sampsi} is verified
  with $\epsilon  = h - \alpha $.
\end{proof}

\rems\ 1. So far the conclusion of Corollary~\ref{finalcor}  was known only for $B$-splines as
generators~\cite{AG00}. For other generators only qualitative results
were known~\cite{AF98} or weak estimates far from the correct sampling
density~\cite{AF98,AG01,Li07}.

2. If $Y= h \bZ$ and $X$ satisfies condition $C_r(\epsilon )$,
then the largest possible gap of consecutive points in $X$ is
$2 h  - 2\epsilon $. Of course, large gaps have to be
compensated by a higher density of neighboring points so that
the condition $|(y_k+\epsilon,y_{k+r}-\epsilon)\cap X'|\ge r+1$ is
still satisfied. By refining the compactness argument in the proof of
Theorem~\ref{sampsi},  it is possible to derive even weaker conditions
for $X$ to be a set of sampling for a shift-invariant space $V_Y(g)$
with totally positive generator $g$.









\def\cprime{$'$} \def\cprime{$'$} \def\cprime{$'$} \def\cprime{$'$}
  \def\cprime{$'$}

\end{document}